\documentclass[11pt, leqno]{amsart}

\usepackage{mathrsfs}
\usepackage{graphicx}

\usepackage{amsfonts,delarray,amssymb,amsmath,amsthm,a4,a4wide}
\usepackage{latexsym}
\usepackage{epsfig}
\usepackage{color}

\newtheorem{thm}{Theorem}[section]
\newtheorem{cor}[thm]{Corollary}

\newtheorem{prop}[thm]{Proposition}

\theoremstyle{definition}

\theoremstyle{remark}
\newtheorem{rem}[thm]{Remark}

\numberwithin{equation}{section}

\newcommand{\R}{\mathbb R}

\newcommand{\e}{\varepsilon}

\newcommand{\p}{\partial}

\newcommand{\comment}[1]{}
\def\h{\hspace*{.24in}} \def\w{\vspace{.12in}}

\usepackage{esvect}

\makeatletter
\@namedef{subjclassname@2020}{%
  \textup{2020} Mathematics Subject Classification}
\makeatother
\begin{document} 

\title[Large dimension behavior of the Hessian eigenvalues]{Large dimension behavior of the Hessian eigenvalues of the unit balls}
\author{Nam Q. Le}
\address{Department of Mathematics, Indiana University, 831 E 3rd St,
Bloomington, IN 47405, USA}
\email{nqle@iu.edu}
\thanks{The author was supported in part by the National Science Foundation under grants DMS-2054686 and DMS-2452320.}

\subjclass[2020]{ 35P15, 35P30, 35J96}
\keywords{Monge--Amp\`ere eigenvalue, Hessian eigenvalue}

\maketitle
\begin{abstract}
We show that a sequence of $k$-Hessian eigenvalues of the unit ball in $\R^n$ stays bounded as long as the ratio $n/k$ stays bounded. Moreover, 
we identify their growth of order at least $(2-\frac{1}{k})$ in $n/k$.
In the case $k=n$, we show that the Monge--Amp\`ere eigenvalues of the unit balls tend to $4$ in the large dimension limit.
  \end{abstract}

\section{Introduction and statement of the main result}

For uniformly convex domains $\Omega$ in $\R^n$ ($n\geq 2$) with  smooth boundaries, Lions \cite{Ln} 
showed that there exist a unique positive constant $\lambda(n;\Omega)$ and a unique (up to positive multiplicative constants) nonzero 
convex function $u\in C^{1,1}(\overline{\Omega})\cap C^{\infty}(\Omega)$ solving the Monge--Amp\`ere eigenvalue problem
\begin{equation}\label{EPLi}
 \left\{
 \begin{alignedat}{2}
   \det D^2 u ~&=[\lambda(n;\Omega)]^n |u|^n \quad~&&\quad\text{in} ~\Omega, \\\
 u&= 0 \quad~&&\quad\text{on}~ \p \Omega.
 \end{alignedat}
 \right.
  \end{equation}
The constant $\lambda(n;\Omega)$ and nonzero convex solutions to \eqref{EPLi} are called the Monge--Amp\`ere eigenvalue and Monge--Amp\`ere eigenfunctions of $\Omega$. In the literature, sometimes 
$[\lambda(n;\Omega)]^n$ is also called the  Monge--Amp\`ere eigenvalue of $\Omega$. Our choice in this note enforces the  Monge--Amp\`ere eigenvalue to have the same scaling as the Laplace eigenvalue. 
\medskip

It is well-known that the first eigenvalue $\lambda(1; n):= \lambda(1; B_1^n)$ of the 
Laplace operator with zero Dirichlet boundary condition on the unit ball 
\[B_1^n:=\{x=(x_1, \cdots, x_n)\in\R^n: x_1^2+\cdots+ x_n^2\leq 1\}\]
of $\R^n$ is equal to the square of the first positive zero $j_{n/2-1}$ of the Bessel function $J_{n/2-1}$. 
By Watson \cite[Section 15.3]{W}, $\lambda (1; n)=[j_{n/2-1}]^2$ grows quadratically in $n$. 
\medskip

We are interested in the large dimension behavior of the Monge--Amp\`ere eigenvalues \[\lambda(n):=\lambda(n; B_1^n)\] of the unit balls in $\R^n$. 
To the best of our knowledge, there are no explicit formulas for the Monge--Amp\`ere eigenvalues and eigenfunctions of any convex domains. Numerical results are available, however, for the Monge--Amp\`ere eigenvalue
of the unit ball in $\R^2$; see, Glowinski--Leung--Liu--Qian  \cite{GLLQ}, Liu--Leung--Qian \cite{LLQ}, and Chen--Lin--Yang--Yi \cite{CLYY}, for example.
Thus, it is not a priori clear if any asymptotic behavior can be obtained for $\lambda(n)$.  
The picture is quite contrast to the Laplace eigenvalue case. 
As will be seen, $\lambda(n)$ tends to $4$ in the large dimension limit.  An asymptotically sharp lower bound for $\lambda(n)$ can be obtained by standard arguments using Euler's first integral. 
For asymptotically sharp upper bounds for $\lambda(n)$, we 
invoke its variational characterizations.

\medskip
For uniformly convex domains $\Omega$ with smooth boundaries, Tso \cite{Ts2} discovered
 the following variational characterization of  $\lambda(n;\Omega)$ using the Rayleigh quotient: 
\begin{multline}
 \label{lam1}
 [\lambda(n;\Omega)]^n=\inf\Bigg\{ \frac{\int_{\Omega} |w|\det D^2 w\,dx}{\int_{\Omega}|w|^{n+1}\,dx}: w\in C^{0,1}(\overline{\Omega})\cap C^2(\Omega)\setminus\{0\},\\w~\text{is convex in } 
 \Omega,~w=0~\text{on}~\p\Omega\Bigg\}.
\end{multline}

The large dimension behavior of $\lambda(n)$ naturally leads to similar investigations concerning 
the Hessian eigenvalues of the unit ball. 

\medskip
Let $1\leq k\leq n$ ($n\geq 2$). We recall the $k$-Hessian eigenvalue problem on uniformly convex domains $\Omega$ in $\R^n$. 
(This problem can be solved in more general domains but our focus here will be on the unit ball.)
For a function $u\in C^2(\Omega)$, let $S_k(D^2 u)$ denote the $k$-th elementary symmetric function of the eigenvalues $\lambda(D^2 u)=(\lambda_1(D^2 u), \cdots,\lambda_n(D^2 u))$ of the Hessian matrix $D^2 u$:
$$S_k(D^2 u)=\sigma_k(\lambda(D^2 u)):=\sum_{1\leq i_1<\cdots<i_k\leq n}\lambda_{i_1}(D^2 u)\cdots \lambda_{i_k}(D^2 u).$$
A function $u\in C^2(\Omega)\cap C(\overline{\Omega})$ is called $k$-admissible if $\lambda(D^2 u)\in \Gamma_k^n$
where $\Gamma_k^n$ is an open symmetric convex cone in $\R^n$, with vertex at the origin, given by
\begin{equation}
\label{Gak}
\Gamma_k^n=\{\lambda=(\lambda_1, \cdots,\lambda_n)\in \R^n\mid \sigma_j(\lambda)>0\quad\text{for all }j=1, \cdots, k\}.
\end{equation}

\medskip
Extending 
the results of Lions \cite{Ln} and Tso \cite{Ts2} from the case $k=n$ to the general case $1\leq k\leq n$, Wang \cite{W1} showed that 
there exist a unique positive constant $\lambda(k;\Omega)$ (called the $k$-Hessian eigenvalue) and a unique (up to positive multiplicative constants) nonzero, $k$-admissible function 
 $u\in C^{\infty}(\Omega)\cap C^{1,1}(\overline{\Omega})$ solving the $k$-Hessian eigenvalue problem
  \begin{equation}
 \label{kEVP_eq}
   S_k (D^2 u)~=[\lambda(k;\Omega)]^k |u|^{k} \h~\text{in} ~\Omega,~
u =0\h~\text{on}~\p \Omega;
\end{equation}
the function $u$ in (\ref{kEVP_eq}) is called a $k$-Hessian eigenfunction of $\Omega$. 
Moreover, Wang \cite{W1} proved the following fundamental variational characterization of $\lambda(k;\Omega)$: 
\begin{multline}
\label{klamR}
 [\lambda(k;\Omega)]^k =\inf\Bigg\{ \frac{\int_{\Omega} |w| S_k(D^2 w)\,dx}{\int_{\Omega} |w|^{k+1}\,dx}: w\in C(\overline{\Omega})\cap C^2(\Omega)\setminus\{0\}, \\ w\text{ is }k\text{-admissible}, w=0~\text{on}~\p\Omega\Bigg\}.
 \end{multline}

From the spectral characterization of $\lambda(k;\Omega)$ in \cite[Theorem 1.1]{L}, we have the following monotonicity property of the Hessian eigenvalues of uniformly convex domains $\Omega$ in $\R^n$:
\begin{equation}
\label{monoineq}
\lambda (n;\Omega)\leq \lambda (n-1;\Omega)\leq \cdots\leq \lambda (2;\Omega)\leq  \lambda (1;\Omega).\end{equation}
So far, there are no sharp estimates available for $k$-Hessian eigenvalues when $1<k<n$. However, for the unit ball, we identify a growth of order at least $(2-\frac{1}{k})$ in $n/k$ of $\lambda (k; B_1^n)$. 
\medskip

Our main result states as follows.
\begin{thm} \label{mainthm} 
 Let $k\leq n$ be positive integers where $n\geq 2$. The following are true.
\begin{enumerate}
\item 
The Monge--Amp\`ere eigenvalue $\lambda(n):=\lambda(n; B_1^n)$ of the unit ball 
$B_1^n$ in $\R^n$ verifies \[4\cdot(2n)^{-1/n} \leq \lambda(n) \leq   4.\]
Consequently,
 \[\lim_{n\to\infty}\lambda(n) =4.\]
\item  
The $k$-Hessian eigenvalue \[\lambda(k; n):=\lambda(k; B_1^n)\] of the unit ball $B_1^n$ in $\R^n$ verifies
\begin{equation*}
\begin{split}
\max\{(n/k-2) (n/k)^{\frac{k-1}{k}}, 4\cdot(2n)^{-1/n}\} &\leq \lambda(k; n)\\& \leq \Big[\frac{(n+ k+1)!}{k! (k+1)! (n-k+1)!}\Big]^{\frac{1}{k}}\leq 2^{\frac{2n+ k+1}{k}}.
\end{split}
\end{equation*}
Consequently, a sequence $\big\{\lambda(k_i; n_i)\big\}_{i=1}^\infty$ (where $1\leq k_i\leq n_i$ are integers) is bounded if and only if $\big\{n_i/k_i\big\}_{i=1}^\infty$ is bounded.
\end{enumerate}
 \end{thm}
 We now indicate some consequences of the main result. 
 
 \medskip
Since the volume of the unit ball in $\R^n$ is $\pi^{n/2}/\Gamma(1+n/2)$, a consequence of Theorem \ref{mainthm} and an extremal property of the Monge--Amp\`ere eigenvalues established in \cite[Theorem 1.4]{LSNS} is the following corollary.
\begin{cor} Let $\Omega$ be a bounded convex domain in $\R^n$. Then
\[[\lambda(n;\Omega)]^n |\Omega|^2 \leq [\lambda(n; B_1^n)]^n|B_1^n|^2 \leq \frac{(4\pi)^n }{[\Gamma (1+n/2)]^2}.\]
\end{cor}
From Theorem \ref{mainthm}, \eqref{monoineq}, and the Stirling's approximation formula, we can extend the limit in Theorem \ref{mainthm} (i) from Monge--Amp\`ere eigenvalues to Hessian eigenvalues as follows.
\begin{cor} \label{lim4cor} Let $1\leq k_i\leq n_i$ be integers such that $n_i\rightarrow\infty$ and $n_i/k_i\rightarrow 1$ when $i\rightarrow\infty$. Then
\[\lim_{i\to\infty}\lambda(k_i; n_i) =4.\]
\end{cor}
\begin{rem} Several remarks are in order.
\begin{itemize}
\item Since $\lambda(1; n)$ grows quadratically in $n$, it would be interesting to improve the growth power $(2-\frac{1}{k})$ of $n/k$ in Theorem \ref{mainthm} (ii) to a quadratic growth in $n/k$ of $\lambda(k; n)$. Note that, for a fixed $k$, $ \Big[\frac{(n+ k+1)!}{k! (k+1)! (n-k+1)!}\Big]^{\frac{1}{k}}$ grows quadratically in $n/k$.
\item Theorem \ref{mainthm} and Corollary \ref{lim4cor} leave us wonder whether
for any converging sequence $\big\{n_i/k_i\big\}_{i=1}^\infty$ (where $1\leq k_i\leq n_i$ are integers and $n_i\rightarrow\infty$ when $i\rightarrow\infty$), the sequence $\big\{\lambda(k_i; n_i)\big\}_{i=1}^\infty$  is also converging. It is still not known whether the sequence $\big\{\lambda([n/2];n)\big\}_{n=1}^\infty$ is converging. 
\item It would be interesting to obtain the asymptotic behaviors,  in large dimensions, of the Monge--Amp\`ere eigenvalues $\lambda(n; \Sigma^n)$ and $\lambda(n; Q^n)$ of the simplex \[\Sigma^n=\{x=(x_1,\cdots, x_n)\in\R^n: x_1+\cdots+ x_n\leq 1, \,\,0\leq x_i\leq 1\text{ for all } i=1,\cdots, n \}\] and the cube 
\[Q^n=\{x=(x_1,\cdots, x_n)\in\R^n: 0\leq x_i\leq 1 \quad\text{for all } i=1,\cdots, n\}.\] 
\end{itemize}
\end{rem}

The convergence result for the Monge--Amp\`ere eigenvalues in Theorem \ref{mainthm} seems to indicate a corresponding convergence result for the Monge--Amp\`ere eigenfunctions. As will be seen in the proof of Theorem \ref{mainthm}, $\lambda(n)$ is asymptotically equal to the $n$th root of 
the Rayleigh quotient  \[\frac{\int_{B_1^n} |w|\det D^2 w\,dx}{\int_{B_1^n}|w|^{n+1}\,dx}\] evaluated at the conical function $|x|-1$. Thus, one may expect that
if $v_n$ is the Monge--Amp\`ere eigenfunction of $B_1^n$ with $\|v_n\|_{L^{\infty}(B_1^n)}=1$, then $\|v_n-(|x|-1)\|_{L^{\infty}(B_1^n)}\rightarrow 0$ when $n\rightarrow \infty$.  However, this is not the case, and we have a gap phenomenon.

\begin{prop}
\label{gapP}
Let $v_n$ be the Monge--Amp\`ere eigenfunction of the unit ball $B_1^n$ in $\R^n$ with $\|v_n\|_{L^{\infty}(B_1^n)}=1$. Then 
\[\limsup_{n \rightarrow \infty}\|v_n-(|x|-1)\|_{L^{\infty}(B_1^n)}\geq \alpha_0,\]
where $\alpha_0\approx 0.1074$ is  the solution of the cubic equation $x^3-3x^2-9x + 1=0$ in $(0, 1)$.
\end{prop}
\begin{rem} After posting this note, the author became aware of the work of Birindelli--Payne \cite{BP} where the authors estimated $\lambda(k; n):=\lambda(k; B_1^n)$ using maximum principle methods which are quite different from our variational arguments. Lemma 5.2 and Theorem 7.1 in \cite{BP} give
\begin{equation}\label{BPkn} \gamma_1(k; n):=2\cdot \Big(  \frac{n!}{k! (n-k)!} \Big)^{\frac{1}{k}}\leq \lambda(k; n) \leq 2 \cdot\Big[\frac{(n-1)!}{k! (n-k)!} \cdot\Big(\frac{n+ 2k}{k+1}\Big)^{k+1}\Big]^{\frac{1}{k}}:=\gamma_2(k; n).\end{equation}
\begin{enumerate}
\item When $k=n$, we have $\gamma_1(n,n)=2$, $\gamma_2(n,n)\to 6$ when $n\to\infty$. It can be verified that $\gamma_2(n, n)\geq 4$ with equality when $n=2$.
\item For a fixed positive integer $k$, when $n\to \infty$, $\gamma_1(k; n)$ grows linearly in $n/k$ while  $\gamma_2(k; n)$ grows quadratically in $n/k$. 
\end{enumerate}
\end{rem}

We will prove Theorem \ref{mainthm}, Corollary \ref{lim4cor}, and Proposition \ref{gapP} in the next section.
\section{Proofs}
In this section, we prove Theorem \ref{mainthm}, Corollary \ref{lim4cor}, and Proposition \ref{gapP}.

\begin{proof}[Proof of Theorem \ref{mainthm}] We prove part (i) in several steps.

\medskip
{\it Step 1.} We first estimate $\lambda(n)$ from below using a convex Monge--Amp\`ere eigenfunction $v\in C^{1,1}(\overline{B_1^n})\cap C^{\infty}(B_1^n) $ of $B_1^n\subset\R^n$ where $v=0$ on $\p B_1^n$. In this case, by its uniqueness up to positive multiplicative constants, $v$ is radial, so we can write $v(x)= u(r)$, where \[r=|x| \quad\text{and }u:[0, 1]\rightarrow (-\infty, 0], \quad u(1)=0,\quad u'(0)=0,\quad u'\geq 0.\] Note that $\det D^2 v(x) = u''(r) (u'(r)/r)^{n-1}$ and 
\[[\lambda(n)]^n= \frac{\int_{B_1^n} |v|\det D^2 v\,dx}{\int_{B_1^n} |v|^{n+1}\,dx}.\]
Using polar coordinates and then integrating by parts, we find
\begin{eqnarray*}[\lambda(n)]^n =\frac{\int_0^1 \int_{\p B_r^n} -u(r) u''(r) (u'(r)/r)^{n-1}\,dS\,dr}{\int_0^1 \int_{\p B_r^n} |u(r)|^{n+1}\,dS\,dr}
&=&\frac{\int_0^1  -u(r) [(u'(r))^{n}]'\,dr}{n\int_0^1 r^{n-1} |u(r)|^{n+1}\,dr} \\&=&\frac{\int_0^1  (u'(r))^{n+1}\,dr}{n\int_0^1 r^{n-1}|u(r)|^{n+1}\,dr}. \end{eqnarray*}
Since $u(1)=0$, we have
$u(r) =-\int_r^1 u'(t) dt$.  Thus, by the H\"older inequality, 
\[|u(r)|^{n+1} \leq \Big ( \int_r^1 |u'(t)|^{n+1} \,dt\Big)  \Big ( \int_r^1 dt\Big)^{n}\leq \Big ( \int_0^1 |u'(t)|^{n+1} \,dt\Big) (1-r)^n. \]
Consequently,
\[\int_0^1 r^{n-1}|u(r)|^{n+1}\,dr \leq  \Big ( \int_0^1 |u'(t)|^{n+1} dt\Big)  \int_0^1 r^{n-1} (1-r)^n \,dr.\]
Hence
\[[\lambda(n)]^n= \frac{\int_0^1  |u'(r)|^{n+1}\,dr}{n\int_0^1 r^{n-1}|u(r)|^{n+1}\,dr} \geq \frac{1}{n  \int_0^1 r^{n-1} (1-r)^n \,dr}.\]
The denominator of the above rightmost term has an Euler's first integral (also known as the Beta function)
which, see Artin \cite[(2.13)]{A},  is equal to  
\[ \int_0^1 r^{n-1} (1-r)^n \,dr = \frac{\Gamma(n)\Gamma (n+1)}{\Gamma (2n+1)}= \frac{(n-1)! n!}{(2n)!},\]
where $\Gamma$ denotes the gamma function.
We have
\begin{equation}
\label{laminf}
[\lambda(n)]^n \geq \frac{(2n)!}{n! n!}={2n\choose n}. 
\end{equation}

{\it Step 2.} We next establish an upper bound on $\lambda(n)$. 
The heuristic idea is to use the merely Lipschitz convex function $w(x)=|x|-1$ in the variational characterization \eqref{lam1} where $\det D^2 w$ should be interpreted as the Monge--Amp\`ere measure of $w$. 
This is possible provided we use a strengthening of Tso's result obtained in \cite{LSNS} which allows us to take convex functions $w$ in \eqref{lam1} with no additional regularity. In particular, when $w(x)=|x|-1$, we have
$\det D^2 w = |B_1^n|\delta_0$ and \eqref{lam1} immediately gives \eqref{lamup}. However, there is currently no  strengthening of Wang's variational characterization of the Hessian eigenvalues, so it is not fully justified to simply use $w(x)=|x|-1$ in \eqref{klamR}. We provide a simple argument that works for all $k$.
Since functions in \eqref{lam1} are required to be in $C^2(B_1^n)$, we can justify this heuristics
 by using in \eqref{lam1} the convex test function \[w(x) =(|x|^2+\e)^{1/2} -(1+\e)^{1/2},\quad\text{where }\e>0.\] 

Note that $w(x) = u(r)$, $r=|x|$, where
$u(r) =(r^2+\e)^{1/2} -(1+\e)^{1/2}$. From  \eqref{lam1}, we deduce that 
\begin{eqnarray*}
[\lambda(n)]^n \leq \frac{\int_{B_1^n} |w|\det D^2 w\,dx}{\int_{B_1^n}|w|^{n+1}\,dx} &=&\frac{\int_0^1  (u'(r))^{n+1}\,dr}{n\int_0^1 r^{n-1}|u(r)|^{n+1}\,dr}\\&=&  \frac{\int_0^1 r^{n+1} (r^2+\e)^{-\frac{n+1}{2}}\, dr}{n\int_0^1 r^{n-1} \big[(1+\e)^{1/2} -(r^2+\e)^{1/2}\big]^{n+1}\,dr}.\end{eqnarray*}
\medskip
Upon letting $\e\rightarrow 0^+$, and then using Euler's first integral, we find
\begin{equation}
\label{lamup}
[\lambda(n)]^n \leq  \frac{1}{n\int_0^1 r^{n-1} (1-r)^{n+1}\,dr} = \frac{\Gamma (2n+2)}{n\cdot \Gamma(n) \cdot\Gamma(n+2)} = \frac{(2n)!}{n!n!}\cdot \frac{(2n+1)}{n+1}.\end{equation}

\medskip
{\it Step 3.} We have the well-known estimates
\begin{equation}\label{C2n}
\frac{2^{2n}}{2n} \leq  \frac{(2n)!}{n! n!} \leq \frac{(2n)!}{n!n!}\cdot \frac{(2n+1)}{n+1}\leq 2^{2n}.\end{equation}
For completeness, we include the simple arguments. 
We have
\[4^n= (1+ 1)^{2n}\geq {2n \choose n-1} + {2n \choose n} + {2n \choose n+1} = \frac{(2n)!}{n!n!}\cdot \frac{(3n+1)}{n+1}, \]
so the last estimate in \eqref{C2n} is obvious. Since the binomial coefficients ${2n \choose 0}$, ${2n \choose 1},\cdots, {2n \choose 2n-1}$, ${2n \choose 2n}$ sum to $2^{2n}$,
to prove the first estimate in \eqref{C2n}, we just observe that the sequence ${2n \choose 0} + {2n \choose 2n}=2, {2n \choose 1}={2n \choose 2n-1}, {2n \choose 2}= {2n \choose 2n-2}, \cdots, {2n \choose n}$ is increasing. 

\medskip
Combining \eqref{laminf} and \eqref{lamup} with \eqref{C2n}, we obtain 
\[4\cdot  (2n)^{-1/n}\leq \lambda(n) \leq   4.\]
The proof of part (i) is complete.

\medskip
We now prove part (ii) concerning $k$-Hessian eigenvalues in several steps. 
\medskip

We will use the following formula (see also \cite[Section 4]{Ts1}): If $w(x)= u(r)$ where $r=|x|$, $u'(0)=0$, $u(1)=0$, and $u'\geq 0$, then
\begin{equation}
\label{krad}
\frac{\int_{B_1^n} |w| S_k(D^2 w)\,dx}{\int_{B_1^n} |w|^{k+1}\,dx} = \frac{k^{-1}{n-1\choose k-1} \int_0^1 r^{n-k} |u'(r)|^{k+1}\,dr}{\int_0^1 r^{n-1} |u(r)|^{k+1}\, dr}.\end{equation}

{\it Step 4.} To bound $\lambda(k; n)$ from above, we use the convex function (which is clearly $k$-admissible) $w(x) =(|x|^2+\e)^{1/2} -(1+\e)^{1/2}$  in \eqref{klamR}, where $\e>0$, and then invoke \eqref{krad}.  One deduces from upon letting $\e\rightarrow 0$ that 
\[[\lambda(k; n)]^k \leq \frac{k^{-1}{n-1\choose k-1} \int_0^1 r^{n-k}\,dr}{\int_0^1 r^{n-1} (1-r)^{k+1}\, dr}= \frac{{n-1\choose k-1}}{k(n-k+1)}\cdot \frac{\Gamma(n+ k+2)}{\Gamma(n)\Gamma(k+2)}= \frac{(n+ k+1)!}{k! (k+1)! (n-k+1)!}.\]
Rewriting the above right-hand side, we find
\[[\lambda(k; n)]^k \leq \frac{1}{n-k+1}{n\choose k} {n+ k+1\choose n} \leq 2^{n}\cdot 2^{n+ k+1}= 2^{2n + k+1}.\]
Thus
\begin{equation}\label{upkn}\lambda(k; n) \leq \Big[\frac{(n+ k+1)!}{k! (k+1)! (n-k+1)!}\Big]^{\frac{1}{k}} \leq 2^{\frac{2n+ k+1}{k}}.\end{equation}

{\it Step 5.} Now, we establish the lower bound 
\begin{equation}
\label{lknlow}
\lambda(k; n)\geq \max\{(n/k-2) (n/k)^{\frac{k-1}{k}}, 4\cdot(2n)^{-1/n}\}.\end{equation}
By \eqref{monoineq} and part (i), we have \[\lambda(k; n)\geq \lambda(n; n)=\lambda(n)\geq 4\cdot(2n)^{-1/n}.\] 
To prove \eqref{lknlow}, it suffices to consider the case $k<n/2$.
 Consider a $k$-Hessian eigenfunction $v\in C^{1,1}(\overline{B_1^n})\cap C^{\infty}(B_1^n) $ of $B_1^n\subset\R^n$ where $v=0$ on $\p B_1^n$. In this case, $v$ is radial so we can write $v(x)= u(r)$ where $r=|x|$ and \[u:[0, 1]\rightarrow (-\infty, 0] \quad\text{with }u(1)=0, ~u'(0)=0, ~\text{and }u'\geq 0.\] 
Since $u(1)=0$, we have
$u(r) =-\int_r^1 u'(t) dt$.  Thus, by the H\"older inequality, 
\begin{eqnarray*}|u(r)|^{k+1} &\leq& \Big ( \int_r^1 t^{n-k}|u'(t)|^{k+1} \,dt\Big)  \Big ( \int_r^1 t^{-\frac{n-k}{k}}\,dt\Big)^{k}\\ &\leq& \Big ( \int_0^1 t^{n-k}|u'(t)|^{k+1} \,dt\Big) \Big(\frac{r^{2-n/k}-1}{n/k-2}\Big)^k. \end{eqnarray*}
Therefore,
\begin{eqnarray*}
\int_0^1 r^{n-1} |u(r)|^{k+1}\, dr \leq  \Big ( \int_0^1 t^{n-k}|u'(t)|^{k+1} \,dt\Big) \int_0^1  r^{n-1}\Big(\frac{r^{2-n/k}-1}{n/k-2}\Big)^k\,dr,
\end{eqnarray*}
which implies, upon recalling \eqref{krad}, that
\begin{eqnarray}\label{lowkn}[\lambda(k; n)]^k= \frac{\int_{B_1^n} |v| S_k(D^2 v)\,dx}{\int_{B_1^n} |v|^{k+1}\,dx}&=& \frac{k^{-1}{n-1\choose k-1} \int_0^1 r^{n-k} |u'(r)|^{k+1}\,dr}{\int_0^1 r^{n-1} |u(r)|^{k+1}\, dr}\nonumber \\&\geq&  \frac{(n/k-2)^k k^{-1}{n-1\choose k-1}}{ \int_0^1  r^{n-1}(r^{2-n/k}-1)^k\,dr}.\end{eqnarray}
Rearranging the denominator of the above rightmost term leads to
\begin{equation}
\label{lknlow2}
\begin{split}[\lambda(k; n)]^k \geq  \frac{(n/k-2)^k k^{-1}{n-1\choose k-1}}{ \int_0^1  r^{2k-1}(1-r^{n/k-2})^k\,dr} \geq \frac{(n/k-2)^k k^{-1} {n-1\choose k-1}}{\int_0^1 r^{2k-1}\, dr}&= 2(n/k-2)^k {n-1\choose k-1}\\&\geq (n/k-2)^k (n/k)^{k-1},\end{split}\end{equation}
where we used that
\[{n-1\choose k-1} = \frac{n-k+1}{1}\cdots \frac{n-1}{k-1}\geq \big(\frac{n}{k}\big)^{k-1}.\]
Hence, \eqref{lknlow} follows. This completes the proof of the theorem.
\end{proof}

\begin{rem}
 Our estimates in \eqref{laminf} and \eqref{lamup} give \[6\leq [\lambda (2)]^2\leq 10,\] while numerical results in Glowinski--Leung--Liu--Qian  \cite[Section 5.3]{GLLQ} and Liu--Leung--Qian \cite[Section 6.1]{LLQ} give the exact value \[[\lambda (2)]^2\approx 7.4897.\]
 Recent numerical results in Chen--Lin--Yang--Yi  \cite[Section 5.1]{CLYY} give the exact value \[[\lambda (2)]^2\approx 7.490039.\]
 We also have
  \[20\leq [\lambda (3)]^3\leq 35.\]
  We assumed $k<n/2$ in deriving \eqref{lowkn}. However, inspecting its proof when $n=3$ and $k=2$, it is still valid in this case.
  When $n=3$ and $k=2$, the right-hand side of  \eqref{lowkn} is $21$.
 Our estimates in \eqref{upkn} and \eqref{lowkn} give \[21\leq [\lambda(2; 3)]^2\leq 30.\]
 It would be interesting to know the exact value of $ [\lambda(2; 3)]^2$.
\end{rem}
\begin{proof}[Proof of Corollary \ref{lim4cor}] For notational simplicity, let \[\lambda_i:= \lambda(k_i; n_i).\] By \eqref{monoineq} and Theorem \ref{mainthm} (i), we have
\[4\cdot (2n_i)^{-1/n_i}\leq \lambda(n_i; n_i)\leq \lambda(k_i; n_i) \leq \lambda(\min(k_i, n_i-[\sqrt{n_i}]-1); n_i).\]
Clearly,
\[\liminf_{i\rightarrow\infty}  \lambda_i\geq 4.\]
It remains to prove 
\[\limsup_{i\rightarrow\infty}  \lambda_i\leq 4.\]
By the limit in Theorem \ref{mainthm} (i), in order to prove the corollary, we can replace $k_i$ by $n_i-[\sqrt{n_i}]-1$ if $k_i\geq n_i-\sqrt{n_i}$. Thus, we can assume that $k_i\leq n_i-\sqrt{n_i}$, and hence \[n_i-k_i\geq \sqrt{n_i}\rightarrow\infty \quad\text{when }i\rightarrow\infty.\]
By the Stirling's approximation formula 
\[n!\approx \sqrt{2\pi n} \Big(\frac{n}{e}\Big)^{n}\Big(1 + O\big(\frac{1}{n}\big)\Big)  \quad\text{when }n\rightarrow\infty,\]
we obtain from Theorem \ref{mainthm} (ii) after some rearrangements and $(k_i+ 1)^{k_i+1}\geq k_i^{k_i}$ that
\begin{eqnarray*}
\limsup_{i\rightarrow\infty}\lambda_i &\leq& \limsup_{i\rightarrow\infty} \Big[\frac{(n_i+ k_i+1)!}{k_i! (k_i+1)! (n_i-k_i+1)!}\Big]^{\frac{1}{k_i}}\\ &\leq&  \limsup_{i\rightarrow\infty}\Bigg[\frac{e\sqrt{n_i+k_i+1}}{2\pi \sqrt{k_i  (k_i+1)}\sqrt{n_i-k_i+1}}\Bigg]^{\frac{1}{k_i}}  \Big[\frac{n_i + k_i +1}{k_i}\Big]^{2} \Big[\frac{n_i+k_i +1}{n_i-k_i +1}\Big]^{\frac{n_i-k_i +1}{k_i}}\\
&=&4,
\end{eqnarray*}
where we recall  $n_i/k_i\rightarrow 1$ when $i\rightarrow\infty$.
The corollary is proved.
\end{proof}

\begin{proof}[Proof of Proposition \ref{gapP}] 
Since $\|v_n\|_{L^{\infty}(B_1^n)}=1$ and $v_n$ is radial, we can write $v_n(x)= u_n(r)$ where  \[r=|x|, \quad u_n:[0, 1]\rightarrow [-1, 0], \quad u_n(0)=-1,\quad u_n(1)=0,\quad u_n'(0)=0,\quad u_n'\geq 0.\] 
Note that, by convexity,
\[1\geq |u_n(r)|\geq 1-r\quad\text{in } [0, 1],\]
and
\[\det D^2 v_n(x) = u_n''(r) (u_n'(r)/r)^{n-1}= [\lambda(n)]^n |u_n(r)|^n.\]
Hence,
\[ \big([u'_n(r)]^{n}\big)' = n [\lambda(n)]^n r^{n-1}|u_n(r)|^n.\]
From $u'_n(0)=0$, we find
\[u'_n(t) =|u'_n(t)| = n^{1/n}\lambda(n) \Big(\int_0^t s^{n-1}|u_n(s)|^n\, ds\Big)^{1/n}\quad\text{for } t\in [0, 1].\]
Integrating the above identity from $0$ to $1$ gives
\begin{equation} \label{int1} n^{1/n}\lambda(n)\int_0^1 \Big(\int_0^t s^{n-1}|u_n(s)|^n\, ds\Big)^{1/n} \,dt=\int_0^1 u_n'(t) \,dt = u_n(1)-u_n(0)=1.\end{equation}
Let $\alpha\in (0, 1)$ be such that
\begin{equation}
\label{adef}
\alpha> \limsup_{n\rightarrow \infty}\|v_n-(|x|-1)\|_{L^{\infty}(B_1^n)} = \limsup_{n\rightarrow \infty} \max_{0\leq r\leq 1}\big(|u_n(r)|-1 +r\big).\end{equation}
Then
\[ |u_n(s)| \leq 1-s+\alpha\quad\text{for all } s\in (0, 1)\text{ and for all large } n.\]
Let 
\[f(s) = s (1+ \alpha-s).\]
Consider $n$ large. We deduce from \eqref{int1}, recalling $|u_n|\leq 1$, that
\begin{equation}
\label{MAODE}
\frac{1}{n^{1/n}\lambda(n)} \leq \int_0^1 \Big(\int_0^t s^{n-1}|u_n(s)|^{n-1}\, ds\Big)^{1/n} \,dt \leq \int_0^1 \|f\|_{L^{n-1}(0, t)}^{\frac{n-1}{n}}\, dt.\end{equation}
Since $0\leq f\leq 2$ on $(0, 1)$, $\|f\|_{L^{n-1}(0, t)}^{\frac{n-1}{n}}$ is bounded from above by $2$, and it tends to $ \sup_{[0, t]} f$ when $n\rightarrow \infty$. Thus, upon letting $n\rightarrow \infty$ in \eqref{MAODE}, Theorem \ref{mainthm} (i) together with the dominated convergence theorem gives 
\[\frac{1}{4} \leq \int_0^1 \max_{0\leq s\leq t} s(1+\alpha-s)\,dt.\]
Observe that
\[\max_{0\leq s\leq t} s(1+\alpha-s)
=\begin{cases}
t(1+\alpha-t)& \text{if }  0\leq t\leq \frac{1+\alpha}{2},\\
\frac{(1+\alpha)^2}{4}& \text{if }   \frac{1+\alpha}{2}\leq t\leq 1.
\end{cases}
\]
We thus have
\[\frac{1}{4} \leq \int_0^{\frac{1+\alpha}{2}} t(1+\alpha-t)\,dt+\int_{\frac{1+\alpha}{2}}^1\frac{(1+\alpha)^2}{4}\, dt = \frac{(1+\alpha)^2 (5-\alpha)}{24}.\]
Rearranging, we find
\[\alpha^3-3\alpha^2-9\alpha + 1\leq 0.\]
Therefore $\alpha\geq \alpha_0$ where $\alpha_0\approx 0.1074$ is the solution in $(0, 1)$ of the cubic equation \[x^3-3x^2-9x + 1=0.\] Now, the definition of $\alpha$ in \eqref{adef} implies the conclusion of the proposition.
\end{proof}

{\bf Acknowledgments.}
The author sincerely thanks the referee for carefully reading the paper and providing constructive comments that help improve the exposition of the paper.

\end{document}